\documentclass[a4paper,11pt,reqno]{amsart}

\usepackage{amsmath,amsthm,amssymb}
\usepackage{mathrsfs}
\usepackage{braket}
\usepackage{multicol}
\usepackage[all]{xy}
\usepackage{comment}
\usepackage{mathtools}
\usepackage{extarrows}

\usepackage{caption}
\usepackage[dvipdfmx]{graphicx}

\usepackage[top=25truemm,bottom=25truemm,left=27truemm,right=27truemm]{geometry}

\numberwithin{equation}{section}
\allowdisplaybreaks

\allowdisplaybreaks

\newtheorem{definition}{Definition}[section]
\newtheorem*{definition*}{Definition}
\newtheorem{proposition}[definition]{Proposition}

\newtheorem{theorem}[definition]{Theorem}
\newtheorem*{theorem*}{Theorem}
\newtheorem{remark}[definition]{Remark}

\newtheorem{fact}[definition]{Fact}
\newtheorem{notation}[definition]{Notation}

\newcommand{\rA}[1]{\Lambda^{A_{N-1}}_{#1}}
\newcommand{\rB}[1]{\Lambda^{B_{N}}_{#1}}

\newcommand{\todaopA}{\mathsf{D}^{A_{N-1}\mathrm{Toda}}}

\newcommand{\todafnA}{f^{A_{N-1}\mathrm{Toda}}}

\title[Branching Formula for $q$-Toda Functions of Type B]
{Branching Formula for $q$-Toda Functions of Type B}
\author{Ayumu~Hoshino}
\author{Yusuke~Ohkubo}
\author{Jun'ichi~Shiraishi}
\address{AH:  Hiroshima Institute of Technology, 2-1-1 Miyake, Hiroshima 731-5193, Japan}
\email{a.hoshino.c3@it-hiroshima.ac.jp}
 \address{YO: Daiichi University of Pharmacy, 
22-1 Tamagawa-cho, Minami-ku, Fukuoka 815-8511, Japan}
 \email{yusuke.ohkubo.math@gmail.com}
\address{JS: Graduate School of Mathematical Sciences, The University of Tokyo, Komaba, Tokyo 153-8914, Japan}
 \email{shiraish@ms.u-tokyo.ac.jp}

\begin{document}

\begin{abstract}
We present a proof of the explicit formula for the asymptotically free eigenfunctions of the $B_N$ $q$-Toda operator which was conjectured by the first and third authors. 
This formula can be regarded as a branching formula from the $B_N$ $q$-Toda eigenfunction restricted to the $A_{N-1}$ $q$-Toda eigenfunctions. 
The proof is given by a contiguity relation of the $A_{N-1}$ $q$-Toda eigenfunctions 
and a recursion relation of the branching coefficients. 
\end{abstract}

\maketitle

\section{Introduction}

Let $f^{A_{N-1}{\rm Toda}}(x|s|q)$ and $f^{B_N{\rm Toda}}(x|s|q)$ 
be the asymptotically free eigenfunctions of 
the $A_{N-1}$ and $B_N$ $q$-Toda operators, respectively 
(Definition \ref{def: toda fn type A}, Definition \ref{def: toda fn type B}). 
Here, $q$ is a generic parameter, and $x=(x_1,\ldots,x_N)$ is an $N$-tuple of variables. 
We introduce an $N$-tuple of continuous parameters (or indeterminates) 
$s=(s_1, \ldots s_N)$, 
while the ordinary $q$-Toda functions contain a weight as a set of discrete parameters. 
A combinatorial explicit formula is known for 
the asymptotically free eigenfunctions of Macdonald's difference operator of type A \cite{Shiraishi2005conjecture,NS2012direct,BFS2014Macdonald}, 
and the one of the $A_{N-1}$ $q$-Toda functions $f^{A_{N-1}{\rm Toda}}(x|s|q)$  can be given by taking a certain limit ($t \rightarrow 0$) of that formula. 
The aim of this paper is to prove the following explicit formula 
for $f^{B_N{\rm Toda}}(x|s|q)$ in terms of $f^{A_{N-1}{\rm Toda}}(x|s|q)$ 
that  was conjectured in \cite{HS2020branching}.

\begin{theorem*}{\bf \ref{thm: branchin rule}.}
The $B_N$ $q$-Toda function $f^{B_N{\rm Toda}}(x|s|q)$ 
is of the form
\begin{align}\label{eq: main result intro}
&f^{B_N{\rm Toda}}(x_1,\ldots,x_N|s_1,\ldots,s_N|q)\nonumber \\
=&
\sum_{\theta=(\theta_1,\ldots,\theta_N)\in \mathbb{Z}^N_{\geq 0}}e^{B_N/A_{N-1}}_{\theta}(s|q) \cdot \prod_{i=1}^N x_i^{-\theta_i} \cdot
f^{A_{N-1}{\rm Toda}}(x_1,\ldots ,x_N|q^{-\theta_1}s_1,\ldots ,q^{-\theta_N}s_N|q),
\end{align}
where we have set
\begin{align}
e^{B_N/A_{N-1}}_{\theta}(s|q):=
&\prod_{k=1}^N
{q^{(N-k+1)\theta_k} \over (q;q)_{\theta_k}(q/s_k^2;q)_{\theta_k}}\\
&\times
\prod_{1\leq i<j\leq N}
{1\over (q s_j/s_i;q)_{\theta_i} (q^{\theta_j-\theta_i} q s_i/s_j;q)_{\theta_i}}
{(q/s_is_j;q)_{\theta_i+\theta_j}\over (q/s_is_j;q)_{\theta_i} (q/s_is_j;q)_{\theta_j}},
\nonumber
\end{align}
$(a;q)_n:=\frac{(a;q)_{\infty}}{(q^na;q)_{\infty}}$, and
$(a;q)_{\infty}:=\prod_{k=1}^{\infty}(1-q^{k-1}a)$. 
\end{theorem*}

The $q$-Toda system has been  studied 
in the connection with representation theory of the quantum groups. 
In particular, the eigenfunctions of the $q$-Toda operators 
can be constructed by Whittaker functions in the Verma module \cite{Etingof1999whittaker,Sevostyanov2000quantum} and 
expressed via  fermionic formulas \cite{FFJMM2009fermionic}. 
Moreover, the $q$-Toda functions are closely related to  characters of Demazure modules 
\cite{GLO2010qDeformed1,GLO2010qDeformed2,GLO2011qDeformed3} and the equivariant K-theory of Laumon spaces \cite{GL2003quantum,BF2005finite}.

The main result  (\ref{eq: main result intro})  can be regarded as a branching rule 
for  the $B_N$ $q$-Toda function restricted to the $A_{N-1}$ $q$-Toda eigenfunctions. 
The proof is given by direct calculation, 
in which  we give a contiguity relation of $f^{A_{N-1}{\rm Toda}}(x|s|q)$ 
(Proposition \ref{prop: contiguity}).
It is an interesting problem to find similar branching formulas for $q$-Toda functions of other types. 

This paper is organized as follows. 
In Section \ref{sec: A B Toda}, 
we recall the definitions of the $q$-Toda functions 
and state the main theorem. 
The proof is given in Section \ref{sec: pf branching rule}.

\section{$A_{N-1}$ and $B_N$ $q$-Toda functions}\label{sec: A B Toda}

First, we recall the aymptotically free eigenfunctions for $A_{N-1}$ $q$-Toda operator. 
Let $q$ be a generic parameter and let $s=(s_1,\ldots, s_N)$ be an $N$-tuple of indeterminates. 
Set
\begin{align}
&\rA{\mathbb{Q}(s,q)}= \mathbb{Q}(s,q)[[x_2/x_1,\ldots, x_N/x_{N-1}]]. 
\end{align}

\begin{definition}
Let $x=(x_1,\ldots, x_N)$. 
The $q$-Toda operator $\todaopA(x|s|q)$ of type A acting on 
$\rA{\mathbb{Q}(s,q)}$
is defined to be
\begin{align}
\todaopA(x|s|q)=\sum_{i=1}^{N-1} s_i (1-x_{i+1}/x_i)T_{q,x_i}+s_N T_{q,x_N}.
\end{align}
Here, $T_{q,x_i}$ is the difference operator defined by 
\begin{align}
T_{q,x_i}f(x_1,\ldots,x_N)=f(x_1,\ldots,qx_i,\ldots,x_N). 
\end{align} 
\end{definition}

The eigenfunctions of $\todaopA(x|s|q)$ is given as follows. 
We use the notation in \cite{HS2020branching}. 

\begin{definition}[\cite{GLO2010qDeformed1}]\label{def: toda fn type A}
Set 
\begin{align}\label{eq: fAN-1}
&f^{A_{N-1}{\rm Toda}}(x|s|q)
= \sum_{\boldsymbol{\theta}\in\mathsf{M}^{(N)}}c^{\rm Toda}_N(\boldsymbol{\theta};s;q)
\prod_{1\leq i<j \leq N} (x_j/x_i)^{\theta_{i,j}}.
\end{align}
Here, $\mathsf{M}^{(N)}=
\{\boldsymbol{\theta}
=(\theta_{ij})_{i,j=1}^{N}| \theta_{ij} \in \mathbb{Z}_{\geq 0}, \theta_{kl}=0 \mbox{ if } k\geq l \}$ is the set of $N\times N$ strictly upper triangular matrices with non-negative integer entries, and the coefficients $c^{\rm Toda}_N(\boldsymbol{\theta};s;q)$ are defined by 
\begin{align}\label{eq: cToda}
&c^{\rm Toda}_N(\boldsymbol{\theta};s;q)\\
&=
\prod_{k=2}^{N}
\prod_{1\leq i\leq j\leq k-1}
{1\over
(q^{\sum_{a=k+1}^N(\theta_{i,a}-\theta_{j+1,a})}qs_{j+1}/s_i;q)_{\theta_{i,k}}}
{q^{\theta_{i,k}}
\over
(q^{\theta_{j,k}-\theta_{i,k}-\sum_{a=k+1}^N(\theta_{i,a}-\theta_{j,a})}q s_i/s_j;q)_{\theta_{i,k}}}.\nonumber
\end{align}
\end{definition}

\begin{fact}[\cite{GLO2010qDeformed1,HS2020branching}]\label{fact: A Toda}
We have 
\begin{align}
\todaopA(x|s|q) \,\, 
\todafnA(x|s|q)=
\sum_{i=1}^N s_i \,\, \todafnA(x|s|q).\label{eq: eigenvalue eq A}
\end{align}
\end{fact}

This formula was originally proved in \cite{GLO2010qDeformed1}. 
A combinatorial explicit formula was given for the asymptotically free eigenfunctions of the Macdonald operator in \cite{Shiraishi2005conjecture,NS2012direct,BFS2014Macdonald}, 
and the formula $f^{A_{N-1}{\rm Toda}}$ can also be directly 
obtained by taking a certain limit $t\rightarrow 0$ of that combinatorial formula \cite{HS2020branching}. 
As for the limit from the Macdonald functions to the Toda functions, 
see also \cite{GLO2010qDeformed1,GLO2010qDeformed2,GLO2011qDeformed3}. 

\begin{notation}
We introduce
\begin{align}
 d_N^{\mathrm{Toda}}((\theta_{i,n})_{1\leq i \leq N-1};(s_i)_{1\leq i \leq N};q)
:=\frac{c^{\mathrm{Toda}}_N((\theta_{i,j})_{1\leq i<j\leq N};(s_i)_{1\leq i \leq N}|q)}
{c^{\mathrm{Toda}}_{N-1}((\theta_{i,j})_{1\leq i<j\leq N-1}; (q^{-\theta_{i,n}}s_i)_{1\leq i \leq N-1}|q)} \quad (N\geq 2). 
\end{align}
Then, $d_N^{\mathrm{Toda}}$ is of the form 
\begin{align}
&d_N^{\mathrm{Toda}}((\theta_i)_{1\leq i \leq N-1};(s_i)_{1\leq i \leq N};q)	\nonumber  \\
& \qquad =\prod_{i=1}^{N-1}
 { 1 \over (q;q)_{ \theta_{i}} }
{q^{\theta_{i}} \over (q s_N/s_i;q)_{\theta_{i}} }
\prod_{1\leq i<j\leq N-1} {1 \over (q s_j/s_i;q)_{\theta_{i}} }
{q^{\theta_i} \over 
	(q^{\theta_{j}-\theta_i+1 }s_i/s_j;q)_{\theta_{i}} }, \label{eq: form of d}
\end{align}
and the $A_{N-1}$ $q$-Toda function can be expressed as 
\begin{align}\label{eq: A_N-1 to A_N-2}
&\todafnA(x|s|q)\\
&=\sum_{\theta =(\theta_1,\ldots , \theta_{N-1}) \in \mathbb{Z}_{\geq 0}^{N-1}} 
 d_N^{\mathrm{Toda}}(\theta ;s;q)	
\prod_{i=1}^{N-1} (x_N/x_i)^{\theta_i}\cdot 
f^{A_{N-2}{\rm Toda}}(x|(q^{-\theta_{i}}s_i)_{1\leq i \leq N-1}|q). \nonumber 
\end{align} 
\end{notation}

Although (\ref{eq: A_N-1 to A_N-2}) follows from the case of the Macdonald functions, 
we can also prove (\ref{eq: A_N-1 to A_N-2}) in a similar manner to Section \ref{sec: pf branching rule}.

Now, we turn to the case of type B. 
Set 
\begin{align}
\rB{\mathbb{Q}(s,q)}=\mathbb{Q}(s,q)[[x_2/x_1,\ldots, x_N/x_{N-1},1/x_N]]. 
\end{align}

\begin{definition}
Define the $B_N$ $q$-Toda operator $\mathsf{D}^{B_N{\rm Toda}}(x|s|q)$ 
acting on $\Lambda^{B_N}_{\mathbb{Q}(s,q)}$ by 
\begin{align}
\mathsf{D}^{B_N{\rm Toda}}(x|s|q)=&
\sum_{i=1}^{N-1}s_i (1-x_{i+1}/x_i)T_{q,x_i}+
s_N (1-1/x_N)T_{q,x_N}\\
& +
s_1^{-1}T_{q,x_1}^{-1}+
\sum_{i=2}^N s_i^{-1} (1-x_i/x_{i-1})T_{q,x_i}^{-1}.\nonumber
\end{align}
\end{definition}

This operator can be obtained by the limit  of the $B_N$ Macdonald operator \cite{HS2020branching}. 
As for the description of the $q$-Toda operators by the quantum groups, 
see \cite{Etingof1999whittaker,Sevostyanov2000quantum,FFJMM2009fermionic}.

\begin{definition}\label{def: toda fn type B}
The asymptotically free eigenfunction $f^{B_N{\rm Toda}}(x|s|q)\in\Lambda^{B_N}_{\mathbb{Q}(s,q)}$ 
of the $B_N$ $q$-Toda operator is defined by 
\begin{align}
&\mathsf{D}^{B_N{\rm Toda}}(x|s|q)f^{B_N{\rm Toda}}(x|s|q)
=\sum_{i=1}^N(s_i+s_i^{-1})f^{B_N{\rm Toda}}(x|s|q),\\
&\left[ f^{B_N{\rm Toda}}(x|s|q) \right]_{x,1}=1. 
\end{align}
Here, $\left[ \quad \right]_{x,1}$ means the constant term with respect to $x_i$'s. 
\end{definition}

Note that $ f^{B_N{\rm Toda}}(x|s|q)$ is uniquely determined. 
We obtain an explicit formula for the $B_N$ $q$-Toda function $f^{B_N{\rm Toda}}(x|s|q)$ 
in terms of the $A_{N-1}$ $q$-Toda functions $f^{A_{N-1}{\rm Toda}}(x|s|q)$. 

\begin{theorem}\label{thm: branchin rule}
The $B_N$ $q$-Toda function $f^{B_N{\rm Toda}}(x|s|q)$ 
satisfies the branching formula
\begin{align}
&f^{B_N{\rm Toda}}(x_1,\ldots,x_N|s_1,\ldots,s_N|q)\nonumber \\
=&
\sum_{\theta=(\theta_1,\ldots,\theta_N)\in \mathbb{Z}^N_{\geq 0}}e^{B_N/A_{N-1}}_{\theta}(s|q) \cdot \prod_{i=1}^N x_i^{-\theta_i} \cdot
f^{A_{N-1}{\rm Toda}}(x_1,\ldots ,x_N|q^{-\theta_1}s_1,\ldots ,q^{-\theta_N}s_N|q),
\label{eq: branchi}
\end{align}
where we have set
\begin{align}\label{eq: e sol}
e^{B_N/A_{N-1}}_{\theta}(s|q):=
&\prod_{k=1}^N
{q^{(N-k+1)\theta_k} \over (q;q)_{\theta_k}(q/s_k^2;q)_{\theta_k}}\\
&\times
\prod_{1\leq i<j\leq N}
{1\over (q s_j/s_i;q)_{\theta_i} (q^{\theta_j-\theta_i} q s_i/s_j;q)_{\theta_i}}
{(q/s_is_j;q)_{\theta_i+\theta_j}\over (q/s_is_j;q)_{\theta_i} (q/s_is_j;q)_{\theta_j}}, 
\nonumber
\end{align}
$(a;q)_n:=\frac{(a;q)_{\infty}}{(q^na;q)_{\infty}}$, and
$(a;q)_{\infty}:=\prod_{k=1}^{\infty}(1-q^{k-1}a)$. 
Note that the constant term of $\todafnA$ is $1$. 
Hence, the constant term of (\ref{eq: branchi}) is also $1$.
\end{theorem}

This formula was conjectured in \cite{HS2020branching}. 
The proof is given in the next subsection.

\begin{remark}
The region of convergence of $\todafnA$ can be derived 
from the case of the Macdonald functions 
(Proposition 6.1 in \cite{NS2012direct}) by taking the limit $t\rightarrow 0$. 
It is an interesting problem to consider the convergence of the formula  (\ref{eq: branchi}).
\end{remark}

\section{Proof of Theorem \ref{thm: branchin rule}}\label{sec: pf branching rule}

In this section, 
we prove Theorem \ref{thm: branchin rule}. 
First we give the following 
relation of the $q$-Toda functions of type A. 

\begin{proposition}\label{prop: contiguity}
The $q$-Toda functions of type A satisfy the contiguity relation
\begin{align}\label{eq: contiguity}
&f^{A_{N-1}{\rm Toda}}(x_1,\ldots , x_{N-1},q x_N|s|q)\\
&=\sum_{k=1}^N(-1)^{N-k}\frac{q^{N-k}\prod_{i=k+1}^{N-1} s_i/s_k  }
{\prod_{i=k+1}^N (1-s_i/s_k)(1-q s_i/s_k)}
(x_N/x_k)
f^{A_{N-1}{\rm Toda}}(x_1,\ldots , x_N|q^{-\varepsilon_k}\cdot s|q). \nonumber 
\end{align}
Here, we used the notation 
\begin{align}\label{eq: varepsilon def}
q^{\pm \varepsilon_i}\cdot s=(s_1,\ldots s_{i-1},q^{\pm 1}s_i,s_{i+1},\ldots, s_N).
\end{align} 
\begin{proof}
First, we show the following equation of the rational functions of $a_i$ and $s_i$: 
\begin{align}\label{eq: pf contiguity}
\prod_{i=1}^{N-1} a_i=\sum_{k=1}^N (s_k/s_N)
\frac{\prod_{i=1}^{N-1}(1-a_i s_k/s_i)}
{\prod_{\substack{1\leq i\leq N \\i\neq k}}(1-s_k/s_i)}.
\end{align} 
Regarding $s_i$'s in the RHS as complex variables, we set 
\begin{align}
F(s):=\sum_{k=1}^{N} (s_k/s_N)\frac{\prod_{i=1}^{N-1}(1-a_i s_k/s_i)}
{\prod_{\substack{1\leq i\leq N \\i\neq k}}(1-s_k/s_i)}
=\sum_{k=1}^{N} \frac{\prod_{i=1}^{N-1}(s_i-a_i s_k)}{\prod_{\substack{1\leq i\leq N \\i\neq k}}(s_i-s_k)}.
\end{align}
For any $\ell=1,\ldots N$, 
the residue at $s_{\ell}=s_{\ell'}$ ($\ell'\neq \ell$) is 
\begin{align}
\mathop{Res}_{s_{\ell}=s_{\ell'}}F(s)\nonumber 
&= \lim_{s_{\ell}\rightarrow s_{\ell'}} F(s)(s_{\ell} -s_{\ell'})\nonumber \\
&=\lim_{s_{\ell}\rightarrow s_{\ell'}}
\left(- \frac{\prod_{i=1}^{N-1}(s_i-a_i s_{\ell})}{\prod_{i\neq \ell, \ell'}(s_i-s_{\ell})} 
+\frac{\prod_{i=1}^{N-1}(s_i-a_i s_{\ell'})}{\prod_{i\neq \ell, \ell'}(s_i-s_{\ell'})} \right)\nonumber \\
&=0. 
\end{align}
Hence $F(s)$ is regular on the whole complex plane with respect to each $s_{\ell}$, 
and it is clear that $F(s)$ is bounded. 
This indicates that $F(s)$ is a constant function.
By the specialization  $s_i = a_{i-1}^{-1} a_{i-2}^{-1} \cdots a_{1}^{-1}s_1$ 
($i=2,\ldots,  N$), 
we have 
\begin{align}
F(s)&=
\sum_{k=1}^{N}(a_{N-1}a_{N-2}\cdots a_{k})
\frac{\prod_{i=1}^{N-1}\left(1-\frac{a_ia_{i-1}\cdots a_1}{a_{k-1}a_{k-2}\cdots a_{1}}\right)}
{\prod_{\substack{1\leq i\leq N\\i\neq k}} \left(1-\frac{a_{i-1}a_{i-2}\cdots a_1}{a_{k-1}a_{k-2}\cdots a_{1}}\right) }\nonumber \\
&=a_{N-1}a_{N-2}\cdots a_{1}. 
\end{align}
This gives (\ref{eq: pf contiguity}).

Substituting $a_i=q^{\theta_i}$ into  (\ref{eq: pf contiguity})
yields 
\begin{align}
\prod_{i=1}^{N-1}q^{\theta_i}=&
\sum_{k=1}^{N-1}(-1)^{N-k}
\frac{q^{N-k}\prod_{i=k+1}^{N-1} s_i/s_k  }
{\prod_{i=k+1}^N (1-s_i/s_k)(1-q s_i/s_k)}
\frac{d_N^{\mathrm{Toda}}(\theta_1,\ldots, \theta_{k}-1,\ldots, \theta_{N-1}|q^{-\varepsilon_k} \cdot s)}
{d_N^{\mathrm{Toda}}(\theta_{1},\ldots , \theta_{N-1}|s)}\nonumber \\
&+
\frac{d_N^{\mathrm{Toda}}(\theta_1,\ldots, \theta_{N-1}|q^{-\varepsilon_N} \cdot s)}{d_N^{\mathrm{Toda}}(\theta_{1},\ldots , \theta_{N-1}|s)}. \label{eq: d_N rel}
\end{align}
By (\ref{eq: A_N-1 to A_N-2}) and (\ref{eq: d_N rel}), 
we obtain the formula (\ref{eq: contiguity}).
\end{proof}
\end{proposition}

\begin{proposition}\label{prop: rec. rel.}
The branching coefficients $e^{B_N/A_{N-1}}_{\theta}(s|q)$ 
satisfy the recursion relation
\begin{align}
&\sum_{i=1}^N \left( (1-q^{-\theta_i})s_i+(1-q^{\theta_i})s_i^{-1}   \right)e^{B_N/A_{N-1}}_{\theta}(s|q) \label{eq: recrel}\\
&=\sum_{k=1}^N
s_N 
\frac{(-1)^{N-k+1}q^{-\theta_N+\delta_{k,N}}q^{N-k} \prod_{i=k+1}^{N-1}(q^{-\theta_i+\theta_k-1}s_i/s_k)}
{\prod_{i=k+1}^{N}(1-q^{-\theta_i+\theta_k -1 } s_i/s_k) (1-q q^{-\theta_i+\theta_k-1} s_i/s_k) }
\, e^{B_N/A_{N-1}}_{(\theta_1\ldots, \theta_k-1,\ldots, \theta_N)}(s|q).
\nonumber 
\end{align}

\begin{proof}
By substituting (\ref{eq: e sol}) into (\ref{eq: recrel}) ,
it can be shown that 
(\ref{eq: recrel}) is equivalent to 
\begin{align}
\sum_{i=1}^{N}\left((1-q^{-\theta_i})s_i+(1-q^{\theta_i})s_i^{-1}\right)
=-\sum_{k=1}^N q^{-\theta_k}s_k
\frac{\prod_{i=1}^N (1-q^{\theta_k}s_i/s_k)(1-q^{\theta_k}/s_is_k ) }
{\prod_{i\neq k} (1-q^{\theta_k-\theta_i}s_i/s_k)(1-q^{\theta_k+\theta_i}/s_is_k ) }. \label{eq: pf branch rul}
\end{align}
By  replacing $q^{\theta_i}$ with generic parameters $Q_i$ 
and shifting $s_i$ to $Q_i s_i$, 
the equation (\ref{eq: pf branch rul}) becomes
\begin{align}
\sum_{i=1}^{N}\left((1-Q_i)s_i+(1-Q_i^{-1})s_i^{-1}\right)
=\sum_{k=1}^N s_k
\frac{\prod_{i=1}^N (1-Q_is_i/s_k)(1-Q_i^{-1}/s_is_k ) }
{\prod_{i\neq k} (1-s_i/s_k)(1-1/s_is_k ) }. \label{type B equiv}
\end{align}
The proof is completed by showing this equation. 
Regarding $s_i$'s as complex variables, we define the function
\begin{align}
F(s)&:=\sum_{k=1}^N s_k
\frac{\prod_{i=1}^N (1-Q_is_i/s_k)(1-Q_i^{-1}/s_is_k ) }
{\prod_{i\neq k} (1-s_i/s_k)(1-1/s_is_k ) }.
\end{align}
A direct calculation shows that 
the residue at 
$s_{\ell}=s_{\ell'}^{\pm 1}$($\ell \neq \ell'$)
is
\begin{align}
\mathop{Res}_{s_{\ell}=s_{\ell'}^{\pm 1}} F(s)
=\lim_{s_{\ell} \rightarrow s_{\ell'}^{\pm 1}}F(s)(s_{\ell}-s_{\ell'})=0
\end{align}
and these singularities are removable. 
Hence, 
$F(s)$ is a regular with respect to each variable $s_{\ell}$ ($\ell=1,\ldots , n$)
on the complex plane except for the origin $0$ (and $\infty$). 
Therefore,  
for arbitrary $\ell$, 
the function $F(s)$ can be given by the Laurent series on $0<|s_{\ell}|<\infty$ 
\begin{align}
F(s)=\sum_{i\in \mathbb{Z}} C_i s_{\ell}^{i},
\end{align}
where $C_i$ is a function of 
$s_1\ldots, s_{\ell-1},s_{\ell+1},\ldots,s_N$. 
Since the orders of the poles at $s_{\ell}=0$ and $s_{\ell}=\infty$ 
are at most $1$, we have $C_{i}=0$ ($i<-1$ or $i>1$). 
It can be shown that 
the residues at $s_{\ell}=0$ and $s_{\ell}=\infty$ are
\begin{align}
&C_{-1}=\mathop{Res}_{s_{\ell}=0} F(s)=1-Q_{\ell}^{-1},\\
&C_{1}=\mathop{Res}_{s_{\ell}=\infty} F(s)=1-Q_{\ell}. 
\end{align}
Therefore, with a constant $\widetilde{C}_0$ independent of $s_{i}$'s,  we can write 
\begin{align}
F(s)=\sum_{i=1}^{N}\left((1-Q_i)s_i+(1-Q_i^{-1})s_i^{-1}\right)+\widetilde{C}_0. 
\end{align} 
Furthermore, we obtain
\begin{align}
\widetilde{C}_0&=F(\sqrt{Q_1}^{-1},\sqrt{Q_2}^{-1}\ldots,\sqrt{Q_N}^{-1})\nonumber \\
&=\sum_{k=1}^N \sqrt{Q_k}^{-1}
\frac{\prod_{i=1}^N (1-\sqrt{Q_iQ_k})(1-\sqrt{Q_k/Q_i}) }
{\prod_{i\neq k} (1-\sqrt{Q_k/Q_i})(1-\sqrt{Q_iQ_k}) }\nonumber \\
&=0. 
\end{align}
This gives (\ref{type B equiv}). 
\end{proof}
\end{proposition}

\begin{proof}[Proof of Theorem \ref{thm: branchin rule}]
The action of $\mathsf{D}^{B_N{\rm Toda}}(x|s|q)$ on the right hand side of  (\ref{eq: branchi})
gives  
\begin{align}\label{eq: act D_B on branchi}
&\mathsf{D}^{B_N{\rm Toda}}(x|s|q) \left(\mbox{RHS of (\ref{eq: branchi})}\right)\nonumber \\
&=
\sum_{\theta\in \mathbb{Z}_{\geq 0}^{N}} 
e^{B_N/A_{N-1}}_{\theta}(s|q) \prod_{i=1}^N x_i^{-\theta_i}\cdot 
\Big\{ \mathsf{D}^{A_{N-1}{\rm Toda}}(x|s|q) -q^{-\theta_N} s_N/x_N T_{q,x_N} \nonumber \\
&\qquad +\mathsf{D}^{A_{N-1}{\rm Toda}}((x_{N-i+1}^{-1})_{i=1}^N|(q^{\theta_{N-i+1}} s_{N-i+1}^{-1} )_{i=1}^N|q)\Big\}
f^{A_{N-1}{\rm Toda}}(x|(q^{-\theta_i} s_{i})_{i=1}^N|q)\nonumber \\
&=
\sum_{\theta\in \mathbb{Z}_{\geq 0}^{N}} 
e^{B_N/A_{N-1}}_{\theta}(s|q) \prod_{i=1}^N x_i^{-\theta_i}\cdot 
\Big\{ \sum_{i=1}^{N} q^{-\theta_i}s_i+\sum_{i=1}^Nq^{\theta_i}s_i^{-1}
 -q^{-\theta_N} s_N/x_N T_{q,x_N} \Big\}\\
&\qquad\qquad   \times f^{A_{N-1}{\rm Toda}}(x|(q^{-\theta_i} s_{i})|q).\nonumber 
\end{align}
Here, 
we used Fact \ref{fact: A Toda} and the symmetry 
\begin{align}
f^{A_{N-1}{\rm Toda}}(x|s|q)
=f^{A_{N-1}{\rm Toda}}((x_{N-i+1}^{-1})_{1\leq i\leq N}|(s_{N-i+1}^{-1} )_{1\leq i \leq N}|q).
\end{align} 
By Proposition \ref{prop: contiguity}, we have 
\begin{align}
&\mathsf{D}^{B_N{\rm Toda}}(x|s|q) \left(\mbox{RHS of (\ref{eq: branchi})}\right)\nonumber \\
&=\sum_{\theta\in \mathbb{Z}_{\geq 0}^{N}} 
e^{B_N/A_{N-1}}_{\theta}(s|q) \prod_{i=1}^N x_i^{-\theta_i}\cdot 
\Big\{ \sum_{i=1}^{N} (q^{-\theta_i}s_i+q^{\theta_i}s_i^{-1})
f^{A_{N-1}{\rm Toda}}(x|(q^{-\theta_l} s_{l})|q)  \nonumber \\
&\qquad \qquad  -q^{-\theta_N}  \sum_{k=1}^N (-1)^{N-k} (s_N/x_k)
\frac{q^{N-k} \prod_{i=k+1}^{N-1}(q^{-\theta_i+\theta_k}s_i/s_k)}
{\prod_{i=k+1}^{N}(1-q^{-\theta_i+\theta_k} s_i/s_k) (1-q q^{-\theta_i+\theta_k} s_i/s_k) }\nonumber \\
&\qquad \qquad\qquad \qquad \qquad \qquad \qquad \qquad \qquad \qquad \qquad 
 \times f^{A_{N-1}{\rm Toda}}(x|q^{-\varepsilon_k}\cdot (q^{-\theta_l} s_{l})_{1\leq l\leq N}|q) 
\Big\} \nonumber \\
&=\sum_{\theta\in \mathbb{Z}_{\geq 0}^{N}} \prod_{i=1}^N x_i^{-\theta_i} \cdot
\Big\{
 \sum_{i=1}^{N} (q^{-\theta_i}s_i+q^{\theta_i}s_i^{-1})e^{B_N/A_{N-1}}_{\theta}(s|q) 
  \\
&\qquad  +\sum_{k=1}^N  s_N
\frac{(-1)^{N-k+1} q^{N-k-\theta_N+\delta_{k,N}} 
\prod_{i=k+1}^{N-1}(q^{-\theta_i+\theta_k-1}s_i/s_k)}
{\prod_{i=k+1}^{N}(1-q^{-\theta_i+\theta_k-1} s_i/s_k) (1-q q^{-\theta_i+\theta_k-1} s_i/s_k) }
e^{B_N/A_{N-1}}_{(\theta_1,\ldots, \theta_k-1,\ldots, \theta_N)}(s|q) \Big\} 
\nonumber \\
&\qquad \qquad\qquad \qquad \qquad \qquad \qquad \qquad \qquad \qquad \qquad 
\times  f^{A_{N-1}{\rm Toda}}(x| (q^{-\theta_l} s_{l})_{1\leq l\leq N}|q),\nonumber 
\end{align}
where we have used that $e^{B_N/A_{N-1}}_{\theta}=0$
if $\theta_j=-1$ for some $j$. 
Proposition \ref{prop: rec. rel.}
shows that this is equal to 
$\sum_{i=1}^{N}(s_i+s_i^{-1})\cdot \left(\mbox{RHS of (\ref{eq: branchi})}\right)$. 
This completes the proof. 
\end{proof}

\section*{acknowledgment}

The authors would like to thank B. Feigin, M. Fukuda,  M. Noumi, and L. Rybnikov
for valuable discussions. 
The researches of J.S. and H.A are partially  supported by JSPS KAKENHI 
(J.S: 19K03512, H.A: 19K03530). 
Y.O. is partially supported by Grant-in-Aid for JSPS Research Fellow 
(18J00754).


\providecommand{\href}[2]{#2}\begingroup\raggedright\endgroup

\end{document}